\documentclass[10pt]{amsart}
\usepackage{amsmath, amsthm}
\usepackage[english]{babel}
\usepackage[T1]{fontenc}
\usepackage[latin1]{inputenc}
\usepackage{amssymb}
\usepackage{amscd}
\usepackage{latexsym}
\usepackage[bookmarksnumbered,plainpages,hypertex]{hyperref}
\usepackage{graphicx}
\usepackage{mathrsfs} 

\newtheoremstyle{theorem}
  {12pt}          
  {12pt}  
  {\sl}  
  {\parindent}     
  {\bf}  
  {. }    
  { }    
  {}     
\theoremstyle{theorem}
\newtheorem{theorem}{Theorem}
\newtheorem{corollary}[theorem]{Corollary}
\newtheorem{remark}[theorem]{Remark}

\newtheorem{lemma}[theorem]{Lemma}

\linethickness{0.5pt}

\newcommand{\ic}{\ensuremath{\mathcal{I}}}

\newcommand{\oc}{\ensuremath{\mathcal{O}}}

\newcommand{\tc}{\ensuremath{\mathcal{T}}}

\newcommand{\Pt}{\mathbb{P}^3}
\newcommand{\PtD}{\mathbb{P}_3^*}

\newcommand{\bZ}{\mathbb{Z}}

\newcommand{\aG}{\alpha}
\newcommand{\bG}{\beta}

\newcommand{\fG}{\varphi}
\newcommand{\eG}{\varepsilon}

\newcommand{\bds}{\begin{displaystyle}}
\newcommand{\eds}{\end{displaystyle}}

\title[Self-linked curves and normal bundle.]{Self-linked curves and normal bundle.}

\author{Ph. Ellia}
\address{Dipartimento di Matematica, 35 via Machiavelli, 44100 Ferrara}
\email{phe@unife.it}


\date{\today}

\begin{document}
\maketitle


\thispagestyle{empty}

\begin{abstract} We give necessary conditions on the degree and the genus of a smooth, integral curve $C \subset \Pt$ to be selk-linked (i.e. locus of simple contact of two surfaces). We also give similar results for set theoretically complete intersection curves with a structure of multiplicity three (i.e. locus of 2-contact of two surfaces).
\end{abstract}

\section*{Introduction.}
The motivation of this note is the following question, raised in \cite{Ha-Po}: doest there exists a smooth, integral curve $C \subset \Pt$, of degree $8$, genus $3$, which is self-linked? We recall that a curve is self-linked if it is the locus of (simple) contact of two surfaces (see Section 1). This question in turn is motivated by the following fact (proved in \cite{Ha-Po}, Proposition 7.5): let $S \subset \Pt$ be a surface with ordinary singularities. Let $C \subset S$ be a smooth, irreducible curve which is the set theoretic complete intersection (s.t.c.i.) of $S$ with another surface. If $C \not \subset Sing(S)$, then $C$ is self-linked (on $S$) (see Remark \ref{R-Ha-Po} for a precise statement). We recall that the problem to know whether or not every smooth irreducible curve $C \subset \Pt$ is a s.t.c.i. is still open. The study of self-linked curves is a first step in this long standing open problem. Self-linked curves have been studied by many authors (see \cite{Ha-Po} and the bibliography therein).

In this note we show that, as expected, no curve of degree $8$, genus $3$ is self-linked. This follows from our main result (Theorem \ref{T-cdt selfL}) which gives necessary conditions on the invariants of a curve to be self-linked. As a consequence we obtain that if $d \geq 13$ and $d>g-3$, then no curve of degree $d$, genus $g$ can be self-linked (Corollary \ref{C-c}).

In the last section we obtain similar results for curves which are set theoretic complete intersections with a triple structure.

Throughout this note we work over an algebraically closed field of characteristic zero.  

\section{Generalities.}

We denote by $C \subset \Pt$ a smooth, irreducible curve of degree $d$, genus $g$. The curve $C$ is \emph{self-linked} if it is (algebraically) linked to itself by a complete intersection $F_a\cap F_b$ of two surfaces of degrees $a, b$. In particular $2d = ab$. This is equivalent to say that there exists a double structure, $C_2$, on $C$ which is a complete intersection of type $(a,b)$.

Let's observe that if $C$ is not a complete intersection, then $C \cap Sing(F_a)\neq \emptyset$ and $C \cap Sing(F_b)\neq \emptyset$. This follows from the fact (see \cite{Ha-Po}, Lemma 7.6) that for a surface $S \subset \Pt$, $Pic(S)/Pic (\Pt )$ is a torsion free abelian group. 

The two surfaces $F_a, F_b$ are tangents almost everywhere along $C$. Moreover at every point $x\in C$ one of the two is smooth (otherwise the embedding dimension of the intersection would be three). So $F_a, F_b$ define a sub-line bundle $L \subset N_C$. Abusing notations $L= N_{C,F_a}= N_{C,F_b}$. The quotient $N_C^* \to L^*\to 0$ defines the double structure $C_2$, hence:
\begin{equation}
\label{eq: L*=I de C dans C1}
0 \to L^* \to \oc _{C_2} \to \oc _C \to 0
\end{equation}

By the exact sequence of liaison:
$$0 \to \ic _{C_2} \to \ic _C \to \omega _C(4-a-b) \to 0$$
we see that $\ic _{C,C_2} \simeq \omega _C(4-a-b)$. This means that $L^* = \omega _C(4-a-b)$. In particular:

\begin{equation}
\label{eq:genre C1}
\deg (L) =: l = d(a+b-4) - 2g + 2
\end{equation}

\begin{remark}
\label{R-Ex1}
If $C$ is a complete intersection, then $C$ is self-linked. If $C$ is a curve on a quadric cone, then $C$ is self-linked. In all these cases $N_C$ splits. 

On the other hand it is easy to give examples of curves which are not self-linked. Let $C \subset \Pt$ be a smooth, irreducible curve whose degree, $d$, is an odd prime number. Assume $h^0(\ic _C(2))=0$. If $C$ is self-linked by $F_a\cap F_b$, then $2d = ab, a \leq b$. Since $d$ is prime, $a=2$, in contradiction with the assumption $h^0(\ic _C(2))=0$.

A less evident fact: if $C \subset \Pt$ is a smooth subcanonical curve (i.e. $\omega _C \simeq \oc _C(a)$ for some $a \in \bZ$) which is not a complete intersection, then $C$ is not self-linked (see \cite{Beo-Ellia}).
\end{remark}

We can add a further class of examples:

\begin{lemma}
\label{L-C on Q}
Let $C$ be a smooth, irreducible curve lying on a smooth quadric $Q \subset \Pt$. If $C$ is not a complete intersection and $\deg (C) > 4$, then $C$ is never self-linked.
\end{lemma}

\begin{proof} Assume $C$ self-linked by $F_a\cap F_b$, $a \leq b$. Let $(\aG ,\bG )$, $\aG < \bG$, denote the bi-degree of $C$ on $Q$. If $F_a = Q$, then $F_b \cap Q$ is a curve of bi-degree $(b,b)=(2\aG , 2\bG )$. It follows that $\aG = \bG$ and $C$ is a complete intersection. So we may assume that $F_a$ is not a multiple of $Q$. The intersection $F_a\cap Q$ consists of $C$ and of curve $A$ of bi-degree $(a-\aG , a-\bG )$. Since $A$ is not empty ($C$ is not a complete intersection) we have $a > \aG$ and $a \geq \bG$. It follows that: $2a > \aG +\bG = d$. So $a > d/2$. Since $ab =2d$, we get $b=2d/a \geq a > d/2$, so $a \leq 3$ hence $d \leq 5$. If $d=5$, then $(a,b)=(2,5)$ in contradiction with $a > d/2$. Hence $d \leq 4$.
\end{proof}

If $d < 5$, then $C$ is rational or elliptic, see Theorem \ref{T-cdt selfL}. This lemma is in contrast with the fact that every curve on a quadric cone is self-linked.

\section{The Gauss map associated to $L \subset N_C$.}

We first recall some constructions associated to a sub-bundle of $N_C$. In what follow we don't assume $C$ self-linked, $C$ is just any smooth, irreducible curve not contained in a plane. If $L$ is a sub-bundle of $N_C$, then $L(-1) \subset N_C(-1)$ comes from a rank two vector bundle: $\tc _L \subset T_{\Pt}(-1)|C$. At each point $x\in C$, $\tc _L(x) \subset T_{\Pt}(-1)(x) = V/d_x$, defines a plane of $\Pt$ containing the tangent line $T_xC$. 

Local computations show that the plane $\tc _L(x)$ is the Zariski tangent plane to the double structure $C_2$ defined by $N_C^* \to L^* \to 0$.

Now the bundle $\tc _L$ defines the Gauss map $\fG _L: C \to D \subset \PtD$ ($\fG _L(x)=\tc _L(x)$). It is known that $\fG _L$ can't be constant and that $D$ can't be a line (\cite{EVdV}, \cite{Hulek-Sacchiero} Theorem 1.6). By the Nakano's exact sequence $\fG _L^*(\oc _{\PtD}(1)) = T_{\Pt}(-1)|C / F_L$, which has degree $d- \deg (F_L)$. Since $L(-1) = \tc _L/T(-1)_C$, we get: 
\begin{equation}
\label{eq:degréfGL}
\deg (\fG _L^*(\oc _{\PtD}(1))) = \deg (\fG _L).\deg (D) = 3d+2g-2-l
\end{equation}

Now consider the dual curve of $D$, $D^* \subset \Pt$ (defined by the osculating planes of $D$). The tangent surface $Tan(D^*)$ is called the \emph{characteristic surface of $L$} and is denoted by $S^\vee _L$. This surface is the envelope surface of the family of planes $\{ \tc _L(x)\} _{x\in C}$. Since the $\tc _L(x)$ are the tangent spaces to the double structure $C_2$, we have $C_2 \subset S^\vee _L$ (see also \cite{Ramella} Lemma 2.1.2).

\noindent If $D$ is a plane curve, then $S^\vee _L$ is the cone over the (plane) dual curve $D^*$.

We will need the following result, which is contained in \cite{Jaffe-On s.t.c.i}: 

\begin{lemma}
\label{L-d=9g=7}
A smooth, integral curve $C \subset \Pt$, of degree $9$, genus $7$ is never self-linked.
\end{lemma}

\begin{proof} If $C$ is self linked it is by a complete intersection of type $(3,6)$. If the cubic surface, $F_3$, is normal, then by (the proof of) Theorem 3.1 in \cite{Jaffe-On s.t.c.i}, we should have $9.6 \leq 6.7$, which is not the case. If the cubic is ruled we conclude with Propositions 3.4, 3.5 of \cite{Jaffe-On s.t.c.i}. Finally if $F_3$ is a cone, it has to be the cone over a smooth cubic curve (see the proof of Theorem 5.1 of \cite{Jaffe-On s.t.c.i}). But a degree $9$ curve on such a cone is a complete intersection $(3,3)$, hence has genus $10$.
\end{proof}  

Now we can state and prove our main result:

\begin{theorem}
\label{T-cdt selfL}
Let $C \subset \Pt$ be a smooth, irreducible curve of degree $d$, genus $g$. Assume $d \geq 5$ and $h^0(\ic _C(2))=0$. If $C$ is self-linked by a complete intersection of type $(a,b)$, then one of the following occurs:

\noindent $g=3, d=6$ and $(a,b)=(3,4)$, or:
\begin{equation}
\label{eq:cdt self}
g \geq 4 \text{ and }4g \geq d(a+b-7)+12
\end{equation}
\end{theorem}

\begin{proof} From (\ref{eq:genre C1}) and \ref{eq:degréfGL} we get  
\begin{equation}
\label{eq: deg fG_L(1) selflinked} 
r:= \deg (\fG ^*_L(\oc _{\PtD}(1)) = \deg (\fG _L).\deg (D) = 4g-4-d(a+b-7)
\end{equation}
Hence we have:
\begin{equation}
\label{eq:pf thm}
4g-4-r = d(a+b-7) \text{ and } 2d = ab.
\end{equation}
The assumption $h^0(\ic _C(2))=0$ implies $b \geq a \geq 3$ and $\deg (D) \geq 3$. Indeed we already know that $\deg (D)\geq 2$. If we have equality, then $C \subset S^\vee _L$ which is a cone over the dual conic $D^*$. So we have: $r\geq 3$.

If $g \leq 1$, $4g-4-d(a+b-7) \geq 3$ implies $a+b \leq 6$, hence $(a,b)=(3,3)$, which is impossible. So $g \geq 2$. If $2 \leq g \leq 3$, we get $(a,b)=(3,4)$, hence $d=6$. Moreover $r=4$ if $g=2$ and $r=8$ if $g=3$.

Assume first that $\fG _L$ is bi-rational. Then $D \subset \PtD$ is an integral curve of degree $r$ and geometrical genus $g$. If $D$ is not contained in a plane, then $g \leq p_a(D) \leq G(r,2)$, where $G(r,2)$ is given by the Halphen-Castelnuovo's bound: $G(r,2)=(r-2)^2/4$ if $r$ is even, $G(r,2)=(r-1)(r-3)/4$, if $r$ is odd. It follows that $g \leq G(7,2)=6$. 
Since $g \geq 2$ we immediately get $r \geq 5$. From what we said above, this implies $g \geq 3$, hence $d \geq 6$. We have $4g-4-r \leq 15$ and from (\ref{eq:pf thm}), since $d \geq 6$, $a+b-7 \leq 2$. It follows that $(a,b;d)=(3,4;6), (4,4;8),(3,6;9), (4,5;10)$. From (\ref{eq:pf thm}) we get: $4(g-1)=r, r+8, r+18, r+20$ and we see that there is no solution with $5 \leq r \leq 7$, $3 \leq g \leq 6$.

In conclusion if $r \leq 7$ and if $\fG _L$ is bi-rational, then $D$ is a plane curve of degree $r$ and geometric genus $g \geq 2$. We have $2 \leq g \leq (r-1)(r-2)/2 = p_a(D)$. Moreover $C_2$ lies on the cone, $K$, over the (plane) dual curve $D^*$. Finally since $\fG _L$ is bi-rational, $C$ is a unisecant on the cone $K$. This implies that $\deg (D^*)+\eG = d\,\,\,(+)$, where $\eG =1,0$, according to whether $C$ passes through the vertex of the cone or not.

Since $g \geq 2$, we get $r \geq 4$. 

If $r =4$ then $2 \leq g \leq 3$ and we already know that $d=6$. If $g=3$, $D$ is smooth and $\deg (D^*)=12$, in contradiction with $(+)$. If $g=2$, $D$ has one double point which can be a node, a cusp or a tacnode. It follows that $\deg (D^*)= 10, 9$ or $8$. In any case we get a contradiction with $(+)$.

If $r=5$, then $2 \leq g \leq 6$ and from (\ref{eq:pf thm}) we get $4g-9=d(a+b-7)$. Since $d \geq 5$, the cases $2 \leq g \leq 3$ are impossible. If $g=4$, the only possibility is $d=7, a+b=8$. Hence $a=b=8$, but then again $d=ab/2=8$: contradiction. In the same way we see that the cases $g=5,6$ are impossible.

If $r=6$ then $2 \leq g \leq 10$ and $4g-10=d(a+b-7)$, with $d=ab/2$. Observe that if $a+b-7=1$, then $a=b=4$ and $d=8$, if $a+b-7=2$, then $(a,b,d)=(3,6,9)$ or $(4,5,10)$. We get that for $g < 10$ the only possibility is $g=7, d=9, (a,b)=(3,6)$, which is excluded by Lemma \ref{L-d=9g=7}. Finally if $g=10$, then $D$ is smooth. It follows that $d = \deg (D^*)+\eG = 30+\eG$. Since (\ref{eq:pf thm}) yields $30=d(a+b-7)$, we get $d=30$ and $a=b=4$, which is impossible.

If $r=7$ then $2 \leq g \leq 15$ and $4g-11 = d(a+b-7)$. For most values of $g\leq 15$, $4g-11$ is a prime number and anyway it always has a simple factorization into prime numbers. Bearing in mind that if $a+b-7=1$, then $a=b=4$ and $d=8$; if $a+b-7=2$ the $(a,b,d) = (3,6,9)$ or $(4,5,10)$ and if $a+b-7 = 3$, then $(a,b,d)=(4,6,12)$, we easily see that there are no solutions.

In conclusion if $r \leq 7$ and $\fG _L$ is bi-rational, then the only possibility is for $r=6$, $d=9$, $g=7$ and $(a,b)=(3,6)$ (in this case $D$ is a plane curve with a triple point).

Now for $3 \leq r \leq 7$, $r = \deg (\fG _L).\deg (D)$ and $\deg (D) \geq 3$, we see that if $\fG _L$ is not bi-rational, then $r=6$, $\deg (\fG _L)=2$ and $\deg (D)=3$.

If $D$ is not contained in a plane it is a twisted cubic. The dual curve $D^*$ is again a twisted cubic and $S^\vee = Tan(D^*)$ is a quartic surface. Since $C_2 \subset S^\vee$, $S^\vee = AF_a+BF_b$. If $b > 4$, it follows that $F_a = S^\vee$, i.e. $a=4$. From (\ref{eq:pf thm}) we get: $4g = d(d-6)/2 + 10$. Since $b=d/2$, $d$ is even, hence $d\equiv 0,2 \pmod{4}$ and we see that the previous equation never gives an integral value for $g$. This shows $b \leq 4$, hence $(a,b,d)= (3,4,6), (4,4,8)$. Plugging these values into (\ref{eq:pf thm}) we get a contradiction.

It follows that $D$ must be a cubic plane curve. If $D$ is smooth (has a node, a cusp), then $\deg (D^*)=6$ ($4$ or $3$). Since $\fG _L$ has degree two, $C$ is a bi-secant on the cone $S^\vee$ over $D^*$. It follows that $d = 2\deg (D^*)+\eG$. Since $C_2 \subset S^\vee$, $S^\vee = AF_a+BF_b$. If $b > \deg (D^*)$, then $F_a=S^\vee$ and $a = \deg (D^*)$. It follows that $b = 2d/\deg (D^*)$. This implies $b=4$. It follows that $(a,b,d)=(3,4,6), (4,4,8)$. Plugging these values into (\ref{eq:pf thm}) we get a contradiction.

In conclusion we must have $r \geq 8$. 
\end{proof}

\begin{remark} Because of Lemma \ref{L-C on Q} the assumption $h^0(\ic _C(2))=0$ is harmless.

There exist smooth curves of degree $6$, genus $3$ which are self-linked (\cite{Gallarati}, \cite{Ellia-63}).

This improves Theorem 7.8 of \cite{Ha-Po}. It follows from (\ref{eq:cdt self}) that no curve of degree $8$, genus $3$ can be self-linked. This answers to a question raised in \cite{Ha-Po} (Introduction and Remark 7.19).
\end{remark}

\begin{corollary}
\label{C-c}
let $C \subset \Pt$ be a smooth, irreducible curve of degree $d > 4$, genus $g$, with $h^0(\ic _C(2))=0$. If $C$ is self-linked, then:
\begin{equation}
\label{eq: g > sqrt 8d}
g \geq \frac{d(\sqrt{8d}-7)}{4}+3
\end{equation}
Moreover if $d \geq 13$ and $d > g-3$ no curve of degree $d$, genus $g$ can be self-linked.
\end{corollary}

\begin{proof} If $2d = ab, a \geq 2$, then $a+b$ varies from $d+2$ ($a=2, b=d$) to $2\sqrt{2d}$ ($a=b=\sqrt{2d}$). The inequality then follows from (\ref{eq:cdt self}).

A curve with $d > g-3$ and $d \geq 13$ cannot lie on a quadric cone. Moreover if $d \geq 13$, then $2d = ab \geq 26$. It follows that $a+b \geq 11$ and inequality (\ref{eq:cdt self} is never satisfied if $d > g-3$.
\end{proof}

\begin{remark}
\label{R-Ha-Po}
A reduced surface $S \subset \Pt$ is said to have \emph{ordinary singularities} if its singular locus consists of a double curve, $R$, the surface having transversal tangent planes at most points of $R$, plus a finite number of pinch points and non-planar triple points. As proved in \cite{Ha-Po}, Proposition 7.5, if a smooth curve is a set theoretic complete intersection on $S$ with ordinary singularities and if $C \not \subset Sing(S)$, then $C$ is self-linked (on $S$).
\end{remark}

\section{Triple structures.}

To conclude let's see how this approach applies also to set theoretic complete intersections (s.t.c.i.) with a triple structure. Assume $F_a \cap F_b = C_3$, where $C_3$ is a triple structure on a smooth, irreducible curve of degree $d$, genus $g$ (i.e. $C_3$ is a locally Cohen-Macaulay (in our case l.c.i.) scheme with $Supp(C_3)=C$ and $ab=3d$). The complete intersection $F_a\cap F_b$ links $C$ to a double structure, $C_2$, on $C$. By liaison we have: $p_a(C_2)-g = d(a+b-4)/2$. Now $C_2$ (which as any double structure on $C$ is a locally complete intersection curve) corresponds to a sub-line bundle $L \subset N_C$. From the exact sequence (\ref{eq: L*=I de C dans C1}), we get:
\begin{equation}
\label{eq: degL triple}
l:= \deg (L) = \frac{d}{2}(a+b-4)-g+1
\end{equation}

\begin{theorem}
\label{T-triple}
Let $C \subset \Pt$ be a smooth, connected curve of degree $d$, genus $g$. Assume $C$ does not lie on a plane nor on a quadric cone. If $C$ is a s.t.c.i. with a triple structure of two surfaces of degrees $a, b$, then:
\begin{equation}
\label{eq: cdt stci triple}
3g \geq \frac{d}{2}(a+b -10) +6
\end{equation}
In particular: $g \geq \dfrac{d}{6}(\sqrt{12d} - 10)+1$.
\end{theorem}

\begin{proof} As before we consider the Gauss map $\fG _L$. By (\ref{eq:degréfGL}) and (\ref{eq: degL triple}), we have:
$$r := \deg (\fG _L).\deg (D) = 3g-3 - \frac{d}{2}(a+b-10).$$
We know that $r \geq 2$ and if equality $C$ lies on a quadric cone. So we may assume $r \geq 3$ and (\ref{eq: cdt stci triple}) follows. For the second inequality, if $ab=3d$, then $a+b \geq 2\sqrt{3d}$.
\end{proof}

Combining with Corollary \ref{C-c} we get:

\begin{corollary}
Let $C \subset \Pt$ be a smooth, connected curve of degree $d$, genus $g$. If $C$ is not contained in a plane nor in a quadric cone and if $g < \dfrac{d(\sqrt{12d}-10)+6}{6}$, then $C$ cannot be a s.t.c.i. with a structure of multiplicity $m \leq 3$.
\end{corollary}

By the way let us observe the following elementary fact:

\begin{lemma} Let $C \subset \Pt$ be a smooth, connected curve of degree $d$, genus $g$. Let $s$ denote the minimal degree of a surface containing $C$. Assume $C$ is the set theoretic complete intersection of two surfaces of degrees $a,b; a \leq b$ and that $a$ is minimal with respect to this property. Let $md=ab$. If $a > s$ or if $h^0(\ic _C(s)) > 1$, then $m \geq d/s^2$.
\end{lemma}

\begin{proof} Assume $C = F_a \cap F_b$ as sets with $a \leq b$ and $ab=md$. If $S \in H^0(\ic _C(s))$, then $S^m \in H^0_*(\ic _X)$, where $X$ denotes the $m-1$-th infinitesimal neighbourhood of $C$ ($\ic _X = \ic _C^m$). It follows that $S^m \in (F_a, F_b)$. So $S^m = AF_a+BF_b$. If $b > sm$, then $S^m = AF_a$ and since $S$ is integral, we get $S^t = F_a$. It follows that $S \cap F_b=C$ as sets. By minimality of $a$, it follows that $F_a=S$. This is excluded by our assumptions ($a > s$ or $h^0(\ic _C(s))>1$). So $b \leq sm$. Thus $m \geq b/s$, hence $m^2 \geq ab/s^2 =md/s^2$ and the result follows.
\end{proof}

Let $C \subset Q$, $Q$ a smooth quadric surface. Assume $C$ is the s.t.c.i. of two surfaces of degrees $a, b$. Then if $d > 3$ and $C$ is not a complete intersection, it is easy to see that $b \geq a > 2$. Hence $m \geq d/4$, where $dm=ab$.




\begin{thebibliography}{SelfL}

\bibitem{Beo-Ellia} Beorchia, V.-Ellia, Ph.: \textit{Normal bundle and complete intersections}, Rend. Sem. Mat. Univers. Politecn. Torino, vol. 48, 4, 553-562 (1990)

\bibitem{EVdV} Eisenbud, D.-Van de Ven, A.: \textit{On the normal bundles of smooth rational space curves}, Math. Ann., {\bf 256}, 453-463 (1981)

\bibitem{Ellia-63} Ellia, Ph.: \textit{Exemples de courbes de $\Pt$ à fibré normal semi-stable, stable}, Math. Ann., {\bf 264}, 389-396, (1983)

\bibitem{Gallarati} Gallarati, D.: \textit{Ricerche sul contatto di superficie algebriche lungo curve}, Mémoire Acad. Roy. Belge {\bf 32}, 1-78 (1960)

\bibitem{Ha-Po} Hartshorne, R.-Polini, C.: \textit{Divisors class groups of singular surfaces}, Preprint arXiv: 1301. 3222v1 [math.AC], 15 Jan 2013 (2013)

\bibitem{Hulek-Sacchiero} Hulek, K.-Sacchiero, G.: \textit{On the normal bundle of elliptic space curves}, Arch. Math., {\bf 40}, 61-68 (1983)

\bibitem{Jaffe-On s.t.c.i} Jaffe, D.: \textit{On set theoretic complete intersections in $\Pt$}, Math. Ann., {\bf 285}, 165-176 (1989)


\bibitem{Ramella} Ramella L.: \textit{Sur les schémas définissant les courbes rationnelles lisses de $\Pt$ ayant fibré normal et fibré tangent restreint fixés}, Mémoire (nouvelle sèrie) de la Soc. Math. de France, {\bf 54}, (1993)



\end{thebibliography}
\end{document}